\documentclass{amsart}
\usepackage{amsthm,latexsym,amsfonts,amsmath,amssymb,mathrsfs,array,manfnt,stmaryrd}
\usepackage[pdftex]{graphicx}
\usepackage[all]{xy}
\usepackage{paralist}
\usepackage{framed}
\usepackage{cancel}
\usepackage{enumitem}
\usepackage{asymptote}
\usepackage{units}
\usepackage{setspace}
\usepackage{amscd}
\usepackage{microtype}
\usepackage{tikz}
\usetikzlibrary{matrix}
\usetikzlibrary{calc}
\usepackage{tikz-cd}

\newcommand{\into}{\hookrightarrow}

\newcommand{\abs}[1]{\left\lvert#1\right\rvert}

\newcommand{\Z}{\ensuremath{\mathbb{Z}}}
\newcommand{\R}{\ensuremath{\mathbb{R}}}
\newcommand{\C}{\ensuremath{\mathbb{C}}}
\newcommand{\Q}{\ensuremath{\mathbb{Q}}}

\newcommand{\A}{\ensuremath{\mathbb{A}}}

\newcommand{\M}{\mathcal{M}}
\newcommand{\Mbar}{\overline{M}}

\renewcommand{\P}{\ensuremath{\mathbb{P}}}

\renewcommand{\bar}[1]{\overline{#1}}

\DeclareMathOperator{\trop}{trop}

\DeclareMathOperator{\Id}{Id}

\DeclareMathOperator{\PGL}{PGL}

\DeclareMathOperator{\pr}{pr}

\DeclareMathOperator{\ev}{ev}

\usepackage{color}

\newcommand{\margincolor}{red}      
\definecolor{darkgreen}{rgb}{0,0.7,0}

\newcounter{margincounter}
\newcommand{\marginnum}{
\ifnum\value{margincounter}<10
\textcolor{\margincolor}{\begin{picture}(0,0)\put(2.2,2.4){\circle{9}}\end{picture}\footnotesize\arabic{margincounter}}
\else\ifnum\value{margincounter}<100
\textcolor{\margincolor}{\begin{picture}(0,0)\put(4.256,2.5){\circle{11}}\end{picture}\footnotesize\arabic{margincounter}}
\else
\textcolor{\margincolor}{\begin{picture}(0,0)\put(6.8,2.5){\circle{14}}\end{picture}\footnotesize\arabic{margincounter}}
\fi\fi
}

\theoremstyle{plain}
\newtheorem{theorem}{Theorem}
\numberwithin{theorem}{section}

\newtheorem{thm}[theorem]{Theorem}

\newtheorem*{thm*}{Theorem}

\newtheorem{prop}[theorem]{Proposition}

\newtheorem{cor}[theorem]{Corollary}

\newtheorem{lem}[theorem]{Lemma}
\newtheorem{claim}[theorem]{Claim}

\theoremstyle{definition}

\theoremstyle{remark}

\newtheorem{rem}[theorem]{Remark}

\newtheorem{ex}[theorem]{Example}

\newtheorem{question}[theorem]{Question}

\usepackage{tikz}
\usetikzlibrary{matrix}
\usepackage{fullpage}
\usepackage{bbm}

\DeclareMathOperator{\CR}{CR}

\usepackage{hyperref}

\makeatletter
\def\subsubsection{\@startsection{subsubsection}{3}%
  \z@{.5\linespacing\@plus.7\linespacing}{-.5em}%
  {\normalfont\bfseries}}
\makeatother

\begin{document}
\title{Cross-ratio degrees and perfect matchings}
\subjclass[2010]{14N10, 14H10, 14T15, 05C30}
\date{\today}
\author{Rob Silversmith}

\begin{abstract}
  \emph{Cross-ratio degrees} count configurations of points
  $z_1,\ldots,z_n\in\P^1$ satisfying $n-3$ cross-ratio constraints, up
  to isomorphism. These numbers arise in multiple contexts in
  algebraic and tropical geometry, and may be viewed as combinatorial
  invariants of certain hypergraphs. We prove an upper bound on
  cross-ratio degrees in terms of the theory of perfect matchings on
  bipartite graphs. We also discuss several of the many perspectives
  on cross-ratio degrees --- including a connection to Gromov-Witten
  theory --- and give many example computations.
\end{abstract}
\maketitle

\section{Introduction}\label{sec:Intro}

\subsection{Cross-ratio degrees and main result} Recall that if
$z_1,z_2,z_3,z_4$ are distinct points on the projective line $\P^1,$
their \emph{cross-ratio} is
$$\CR(z_1,z_2,z_3,z_4)=\frac{(z_3-z_1)(z_4-z_2)}{(z_3-z_2)(z_4-z_1)}\in\C\setminus\{0,1\},$$
extended as usual in case one of $z_1,z_2,z_3,z_4$ is $\infty.$ This
paper is concerned with counting configurations of points in $\P^1$
satisfying constraints on their cross-ratios, as follows. Let $n\ge4$,
and for each $i=1,\ldots,n-3$, fix a 4-tuple
$e_i=(v_{i,1},\ldots,v_{i,4})$ of distinct elements of
$\{1,\ldots,n\}.$ Write $\mathcal{T}=(e_1,\ldots,e_{n-3}).$ Then given
distinct points $z_1,\ldots,z_n\in\P^1,$ we obtain a tuple of
cross-ratios:
$$\CR_{\mathcal{T}}(z_1,\ldots,z_n)=\left(\CR(z_{v_{1,1}},z_{v_{1,2}},z_{v_{1,3}},z_{v_{1,4}}),\ldots,\CR(z_{v_{n-3,1}},z_{v_{n-3,2}},z_{v_{n-3,3}},z_{v_{n-3,4}})\right)\in(\C\setminus\{0,1\})^{n-3}.$$
Recall that $M_{0,n}$ is the $(n-3)$-dimensional variety that
parametrizes $n$-tuples $(z_1,\ldots,z_n)\in\P^1$ of distinct points,
up to M\"obius transformation. As cross-ratios are invariant under
M\"obius transformations, $\CR_{\mathcal{T}}$ defines a map
$M_{0,n}\to(\C\setminus\{0,1\})^{n-3}.$ % Alternatively,
% $\CR_{\mathcal{T}}$ can be naturally identified with the product of
% forgetful maps
% $\mu_{e_1}\times\cdots\times\mu_{e_{n-3}}:M_{0,n}\to(M_{0,4})^{n-3}$,
% where $\mu_{e}:M_{0,n}\to M_{0,\abs{e}}$ forgets all marked points
% except those in $e$.
Since $\dim(M_{0,n})=\dim((\C\setminus\{0,1\})^{n-3})=n-3,$
$\CR_{\mathcal{T}}$ has a well-defined degree $d_{\mathcal{T}}$, which
we call the \textbf{\textit{cross-ratio degree}} of $\mathcal{T}$, a
nonnegative integer. Our main result, Theorem \ref{thm:Matchings} just
below, gives an upper bound for $d_{\mathcal{T}}$ in terms of matching
theory on bipartite graphs.

Permuting the 4 elements of $e_i$ corresponds to post-composing
$\CR_{\mathcal{T}}$ by a change of coordinates on the $i$-th copy of
$\C\setminus\{0,1\}$, which does not affect $d_{\mathcal{T}}.$
Furthermore, permuting the sets $e_i$ corresponds to post-composing
$\CR_{\mathcal{T}}$ by the action of $S_{n-3}$ on
$(\C\setminus\{0,1\})^{n-3}$; this again does not affect
$d_{\mathcal{T}}.$ Finally, renaming the elements of $\{1,\ldots,n\}$
clearly does not affect $d_{\mathcal{T}}.$ These observations show
that $d_{\mathcal{T}}$ depends only on the isomorphism class of
$\mathcal{T}=(V,E)$ as a hypergraph with vertices $V$ and hyperedges
$E$, where $\abs{V}=n,$ $\abs{E}=n-3,$ and $\abs{e}=4$ for all
$e\in E.$ (We say $\mathcal{T}$ is a \emph{4-uniform} hypergraph.)

Cross-ratio degrees may be computed individually in many ways, see
Section \ref{sec:AlternateMethods} just below. Perhaps the most
notable of these is a recursion due to Goldner (\cite{Goldner2021},
see Section \ref{sec:Goldner}), from which any cross-ratio degree
can be recovered from the base case $d_{\{\{1,2,3,4\}\}}=1$. It is
desirable, and more challenging, to compute $d_{\mathcal{T}}$ in a way
that indicates a clear combinatorial relationship with the structure
of the hypergraph $\mathcal{T}$.

% In this paper we do something slightly weaker; we prove an
% \emph{upper bound} for $d_{\mathcal{T}}$ in terms of matching
% theory.
Recall that a hypergraph is exactly characterized by its bipartite
incidence graph $G_{\mathcal{T}}=(V\sqcup E,I)$, i.e. the bipartite
graph whose vertices are $V\sqcup E,$ and whose edges are the
\emph{incidence set} $I=\{(v,e)\in V\times E:v\in e\}.$ Recall that a
\emph{perfect matching} in a bipartite graph $G=(A\sqcup B,I)$ is a
bijection $f:A\to B$ with $(a,f(a))\in I$ for all $a\in A.$ We write
$P(G)$ for the number of perfect matchings in $G$.
\begin{thm}\label{thm:Matchings}
  Let $\mathcal{T}$ be a 4-uniform hypergraph with $n$ vertices and
  $n-3$ hyperedges. Then for any three distinct vertices
  $\{v_1,v_2,v_3\}\subset V,$ the cross-ratio degree of $\mathcal{T}$
  satisfies
  $$d_{\mathcal{T}}\le P(G_{\mathcal{T}}-\{v_1,v_2,v_3\}),$$ where
  $G_{\mathcal{T}}-\{v_1,v_2,v_3\}$ is the bipartite graph obtained by
  deleting $v_1,v_2,v_3$.
\end{thm}
Note that both partite sets of $G_{\mathcal{T}}-\{v_1,v_2,v_3\}$ have
cardinality $n-3,$ hence this graph may admit perfect matchings.
Note that $P(G_{\mathcal{T}}-\{v_1,v_2,v_3\})$ is also equal to the
number of systems of distinct representatives in the hypergraph
$\mathcal{T}-\{v_1,v_2,v_3\},$ as well as the permanent of the square
matrix obtained by deleting columns $v_1,v_2,v_3$ from the biadjacency
matrix of $G_{\mathcal{T}}$.

To prove Theorem \ref{thm:Matchings}, we first reformulate
$d_{\mathcal{T}}$ as a curve-counting invariant\footnote{In fact,
  $d_{\mathcal{T}}$ is a \emph{Gromov-Witten invariant}, see Corollary
  \ref{cor:GromovWitten} and Section \ref{sec:GromovWitten}
  below.}. For each $v\in V,$ fix a general \emph{linear subvariety}
$Y_v\subseteq(\P^1)^{n-3}$ of type $E_v$ (see Section
\ref{sec:Reformulation}).

\medskip

\noindent\textbf{Theorem \ref{thm:TurnIntoCurveCounting}.}  Exactly
$d_{\mathcal{T}}$ rational curves $C\subseteq(\P^1)^{n-3}$ of
multidegree $(1,\ldots,1)$ intersect $Y_v$ for all $v\in V.$

\medskip

% The proof of Theorem \ref{thm:Matchings} is summarized as follows. In
% Section \ref{sec:Reformulation}, we reformulate $d_{\mathcal{T}}$ as a
% curve-counting invariant (Theorem
% \ref{thm:TurnIntoCurveCounting}). Specifically, $d_{\mathcal{T}}$
% counts smooth curves in $(\P^1)^{n-3}$ of multidegree $(1,\ldots,1)$
% satisfying a collection of incidence conditions derived from
% $\mathcal{T}$.
In Section \ref{sec:Proof}, we translate this curve-counting problem
into a problem of intersecting \emph{incidence subvarieties} of a
moduli space $\M$, and compute the ``answer'' via a calculation in the
cohomology ring of a compactification $\overline{\M}\cong(\P^3)^{n-4}$
of $\M$. This ``answer'' is an upper bound only, because the boundary
$\overline{\M}\setminus\M$ may contribute nontrivially (Example
\ref{ex:NonzeroDiscrepancy}).

Proving Theorem \ref{thm:Matchings} requires three steps. First, for
each incidence subvariety $Z_v\subseteq\M$, we compute the class of
its Zariski closure $[\bar{Z_v^e}]\subseteq\overline{\M}$, by a fairly
direct calculation in the cohomology ring of $(\P^3)^{n-4}$
(Proposition \ref{prop:CalculateClass}). Second, we use Kleiman's
Bertini theorem \cite{Kleiman1974} to prove that $d_{\mathcal{T}}$ is
bounded above by the intersection product $\prod_v[\bar{Z_v^e}]$
(Claim \ref{claim:UpperBound}).  Third, we show combinatorially that
$\prod_v[\bar{Z_v^e}]$ is equal to a count of perfect matchings (Claim
\ref{claim:Matchings}).

We discuss some combinatorial implications of Theorem
\ref{thm:Matchings} in Section \ref{sec:Analysis}; in particular,
Proposition \ref{prop:SurplusUpperBound} gives a connection to the
classical notion of \emph{graph surplus} in matching theory.

\begin{rem}\label{rem:Histogram}
  The upper bound in Theorem \ref{thm:Matchings} appears to often give
  a good estimate for $d_{\mathcal{T}}$. Figure \ref{fig:histogram}
  shows histograms of
  $\delta_{\mathcal{T}}:=d_{\mathcal{T}}-\min_{v_1,v_2,v_3\in
    V}P(G_{\mathcal{T}}-\{v_1,v_2,v_3\})$, as $\mathcal{T}$ ranges
  over several thousand randomly-generated 4-uniform hypergraphs such
  that $\min_{v_1,v_2,v_3\in V}P(G_{\mathcal{T}}-\{v_1,v_2,v_3\})>0$,
  where $n=10$ and $n=15$. Note that when $n=10$ the upper bound is
  equal to $d_{\mathcal{T}}$ around $\approx\!75\%$ of the time (and
  $\approx\!32\%$ when $n=15$). In these examples, the average value
  of $d_{\mathcal{T}}$ was $\approx\!1.9$ when $n=10$, and
  $\approx\!7.3$ when $n=15$.
  \begin{figure}
    \centering
    \includegraphics[height=1.5in]{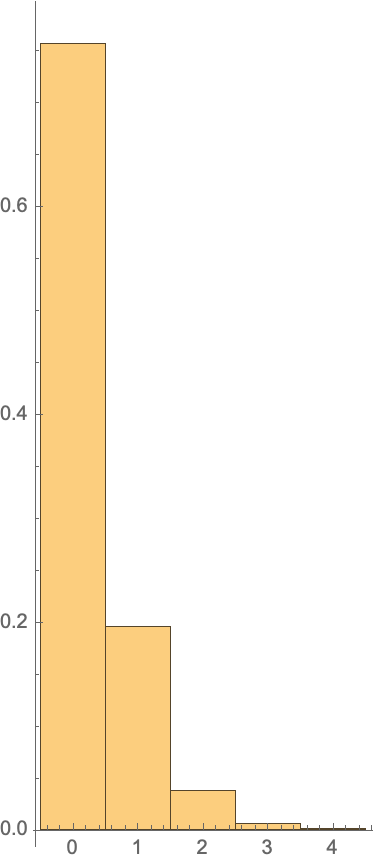}\quad\quad\quad
    \includegraphics[height=1.5in]{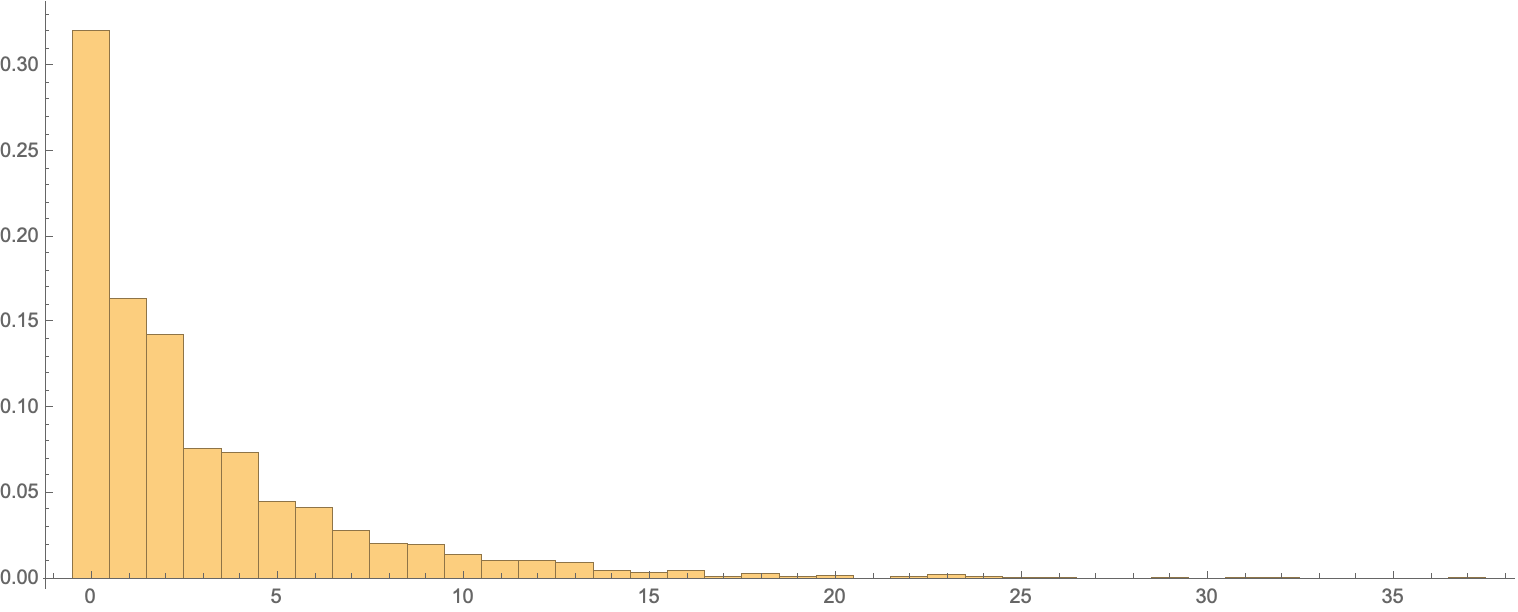}
    \caption{Discrepancies between $d_{\mathcal{T}}$ and the upper
      bound, with $n=10$ (left) and $n=15$ (right)}
    \label{fig:histogram}
  \end{figure}
\end{rem}

While many combinatorial aspects of $M_{0,n}$ and $\Mbar_{0,n}$ are
well-studied (e.g. the characterization of the cohomology groups of
$\Mbar_{0,n}$ in terms of marked trees
\cite{Keel1992,KontsevichManin1994}), this is to our knowledge the
first appearance of the rich combinatorial theory of matchings. From
the perspective of Gromov-Witten theory (see Section
\ref{sec:GromovWitten} just below), the result is fairly unusual in
that proving upper bounds for classes of Gromov-Witten invariants is
difficult in general. In the reverse direction, Theorem
\ref{thm:Matchings} motivates studying cross-ratio degrees from a
combinatorial perspective as interesting hypergraph invariants.

\subsection{Alternate methods of computing
  \texorpdfstring{$d_{\mathcal{T}}$}{d\_T}}\label{sec:AlternateMethods}

We now discuss several of the many ways of computing cross-ratio degrees. As
these are not our main focus, we do not include every detail.

\subsubsection{Coordinate computation}
In small cases, the cross-ratio constraint equations may be solved
directly:
\begin{ex}\label{ex:CoordinateCalculation}
  We now describe the simplest hypergraph $\mathcal{T}$ with
  $d_{\mathcal{T}}\ne0,1.$ It will be a running example in this
  paper; we will show $d_{\mathcal{T}}=2$ in six different ways in
  this section! Let $\mathcal{T}=(V,E)$ with $V=\{1,\ldots,6\}$ and
  $E=\{\{1,2,3,4\},\{1,2,5,6\},\{3,4,5,6\}\},$ i.e. $G_{\mathcal{T}}$
  has biadjacency matrix $$
  \begin{pmatrix}
    1&1&1&1&0&0\\
    1&1&0&0&1&1\\
    0&0&1&1&1&1
  \end{pmatrix}
  .$$ We will calculate $d_{\mathcal{T}}$ by counting 6-marked lines
  $(\P^1,p_1,\ldots,p_6)$ satisfying
  \begin{align*}
    \CR(p_1,p_2,p_3,p_4)&=a_1&
    \CR(p_1,p_2,p_5,p_6)&=a_2&
    \CR(p_3,p_4,p_5,p_6)&=a_3,
  \end{align*}
  where $a_1,a_2,a_3$ are generically chosen scalars. We choose
  coordinates on $\P^1$ so that $p_1=\infty,$ $p_2=0$, $p_3=1.$ Then
  the equations above simplify to
  \begin{align*}
    p_4&=a_1&\frac{p_6}{p_5}&=a_2&\frac{(p_5-1)(p_6-p_4)}{(p_6-1)(p_5-p_4)}&=a_3.
  \end{align*}
  There are two solutions $(p_4,p_5,p_6)=(a_1,b,a_2b),$ where $b$ is a
  solution to the nondegenerate quadratic equation
  $a_2(1-a_3)b^2+(a_1a_2a_3-a_1-a_2+a_3)b+a_1(1-a_3)=0.$ Thus
  $d_{\mathcal{T}}=2.$
\end{ex}

\subsubsection{Geometric computation via Theorem
  \ref{thm:TurnIntoCurveCounting}} Theorem
\ref{thm:TurnIntoCurveCounting} says that $d_{\mathcal{T}}$ is equal
to the number of multidegree-$(1,\ldots,1)$ curves in $(\P^1)^{n-3}$
satisfying $n$ incidence conditions derived from $\mathcal{T}$, where
$n=\abs{V}$. There are various ways to compute these curve counts in
individual cases, e.g. we now see that in our running example, we may
compute $d_{\mathcal{T}}$ by carefully specializing the incidence
conditions:

\begin{ex}\label{ex:cube}
  Let $\mathcal{T}$ be the hypergraph of Example
  \ref{ex:CoordinateCalculation}.  By Theorem
  \ref{thm:TurnIntoCurveCounting}, $d_{\mathcal{T}}$ is equal to the
  number of $(1,1,1)$-curves $C\subseteq(\P^1)^3$ that intersect
  \begin{itemize}
  \item two general lines $Y_1,Y_2$ of the form $(a,b,*)$,
  \item two general lines $Y_3,Y_4$ of the form $(a,*,b)$, and
  \item two general lines $Y_5,Y_6$ of the form $(*,a,b)$.
  \end{itemize}
  (Here e.g. $(a,*,b)$ means the first and third coordinates are fixed
  and the second coordinate is unconstrained.) We solve this
  intersection problem by specializing the incidence conditions,
  omitting some straightforward details. We will need to allow
  singular $(1,1,1)$-curves (connected, as always). We choose
  \begin{align*}
    Y_1&=(0,0,*)&Y_3&=(0,*,0)&Y_5&=(*,\infty,0)\\
    Y_2&=(\infty,0,*)&Y_4&=(\infty,*,\infty)&Y_6&=(*,\infty,\infty).
  \end{align*}
  Suppose $C$ is a $(1,1,1)$-curve passing through
  $Y_1,\ldots,Y_6$. As $C$ intersects $Y_1$ and $Y_2$, which both lie
  in the plane $(*,0,*)$, $C$ must have an irreducible component $C_1$
  contained in this plane. Similarly, as $C$ intersects $Y_5$ and
  $Y_6$, $C$ must have an irreducible component $C_2$ contained in the
  plane $(*,\infty,*).$ This implies that we have multidegrees
  $\deg(C_1)=(1,0,0)$ and $\deg(C_2)=(0,1,0)$, and that $C$ has a
  third irreducible component $C_3$ with degree $(0,0,1).$ The
  component $C_3$ is uniquely determined by $C_1$ and $C_2$ (and
  connectedness of $C$). There are exactly two curves of this form
  that pass through $Y_3$ and $Y_4$, given by the choices
  \begin{align*}
    (C_1,C_2)&=((*,0,0),(\infty,\infty,*))&&\text{and}&(C_1,C_2)&=((*,0,\infty),(0,\infty,*)).
  \end{align*}
  One can (and must) confirm that both curves contribute to the
  intersection with multiplicity 1, giving $d_{\mathcal{T}}=2$.
\end{ex}

\subsubsection{Tropical coordinate computation}\label{sec:TropicalCoordinate}
One may compute $d_{\mathcal{T}}$ via tropical geometry, by
tropicalizing $\CR_{\mathcal{T}}$. Recall that $M_{0,n}^{\trop}$ is
the polyhedral complex parametrizing finite metric trees with $n$
marked (infinite-length) half-edges, where each vertex has valence at
least 3. The map $\CR_{\mathcal{T}}$ has a tropicalization
$\CR_{\mathcal{T}}^{\trop}:M_{0,n}^{\trop}\to(M_{0,4}^{\trop})^{n-3},$
and it follows from algebraic-tropical correspondence theorem of
Tyomkin \cite[Thm. 5.1]{Tyomkin2017} that this map is ``generically
finite of degree $d_{\mathcal{T}}$'' in the sense that over a dense
open subset of $(M_{0,4}^{\trop})^{n-3}$, fibers consist of
$d_{\mathcal{T}}$ points counted with multiplicity. Explicitly, we
have the following.

\begin{ex}\label{ex:tropical}
  Let $\mathcal{T}$ be the hypergraph of Example
  \ref{ex:CoordinateCalculation}.  Consider the cone in
  $(M_{0,4}^{\trop})^3$ consisting of triples of tropical curves of
  types
  \begin{align*}
    \left(\raisebox{-10pt}{
    \begin{tikzpicture}[scale=.5]
      \draw (0,0) node {$\bullet$};
      \draw (1,0) node {$\bullet$};
      \draw (0,0)--(1,0);
      \draw (0,0)--(-.5,.5);
      \draw (0,0)--(-.5,-.5);
      \draw (1,0)--(1.5,.5);
      \draw (1,0)--(1.5,-.5);
      \draw (-.5,.5) node[left] {1};
      \draw (-.5,-.5) node[left] {2};
      \draw (1.5,.5) node[right] {3};
      \draw (1.5,-.5) node[right] {4};
    \end{tikzpicture}
    }
    ,\raisebox{-10pt}{
    \begin{tikzpicture}[scale=.5]
      \draw (0,0) node {$\bullet$};
      \draw (1,0) node {$\bullet$};
      \draw (0,0)--(1,0);
      \draw (0,0)--(-.5,.5);
      \draw (0,0)--(-.5,-.5);
      \draw (1,0)--(1.5,.5);
      \draw (1,0)--(1.5,-.5);
      \draw (-.5,.5) node[left] {1};
      \draw (-.5,-.5) node[left] {2};
      \draw (1.5,.5) node[right] {5};
      \draw (1.5,-.5) node[right] {6};
    \end{tikzpicture}
    }
    ,\raisebox{-10pt}{
    \begin{tikzpicture}[scale=.5]
      \draw (0,0) node {$\bullet$};
      \draw (1,0) node {$\bullet$};
      \draw (0,0)--(1,0);
      \draw (0,0)--(-.5,.5);
      \draw (0,0)--(-.5,-.5);
      \draw (1,0)--(1.5,.5);
      \draw (1,0)--(1.5,-.5);
      \draw (-.5,.5) node[left] {3};
      \draw (-.5,-.5) node[left] {4};
      \draw (1.5,.5) node[right] {5};
      \draw (1.5,-.5) node[right] {6};
    \end{tikzpicture}
    }
\right)
  \end{align*}
  This cone is isomorphic to $(\R_{\ge0})^3$ by recording the edge
  lengths, and its link is a filled triangle. One may compute that the
  preimage of this triangle under $\CR_{\mathcal{T}}^{\trop}$ is as
  depicted in Figure \ref{fig:tropicalization}. Some fibers contain
  $d_{\mathcal{T}}=2$ points, while others contain a single point. The
  latter points map with multiplicity 2, as the associated
  piecewise-linear map is locally given by
  $(x,y,z)\mapsto(x+y,y+z,x+z)$, whose determinant is 2. (The real
  numbers $(x,y,z)$ are the edge lengths of the metric tree.)
  \begin{figure}
    \centering
    \begin{tikzpicture}[xscale=2]
      \filldraw[color=black,fill=black,fill opacity=.1] (0,0)--(2,0)--(1,1)--cycle;
%      \draw (1,0)--(1.5,.5)--(.5,.5)--cycle;
      \draw (0,1.4)--(1,1.5)--(2,1.4)--(1.5,2)--(.9,2.5)--(.5,2)--cycle;
      \draw (1,1.5)--(1.5,2)--(.5,2)--cycle;
      \filldraw[color=black,fill=white,fill opacity=.8] (1,1.5)--(0,1.7)--(.5,2);
      \filldraw[color=black,fill=white,fill opacity=.8] (1,1.5)--(2,1.7)--(1.5,2);
      \filldraw[color=black,fill=white,fill opacity=.8] (.5,2)--(1.1,2.5)--(1.5,2);
      \draw[->] (1,1.4)--(1,1.1);
      \draw[blue] (3,.7) node {\begin{tikzpicture}[scale=.4]
          \draw (0,0) node {$\bullet$};
          \draw (1,0) node {$\bullet$};
          \draw (2,0) node {$\bullet$};
          \draw (3,0) node {$\bullet$};
          \draw (0,0)--(3,0);
          \draw (0,0)--(-.5,.5);
          \draw (0,0)--(-.5,-.5);
          \draw (1,0)--(1,0.7);
          \draw (2,0)--(2,0.7);
          \draw (3,0)--(3.5,.5);
          \draw (3,0)--(3.5,-.5);
          \draw (-.5,.5) node[left] {1};
          \draw (-.5,-.5) node[left] {2};
          \draw (1,0.7) node[above] {4};
          \draw (2,0.7) node[above] {3};
          \draw (3.5,.5) node[right] {5};
          \draw (3.5,-.5) node[right] {6};
        \end{tikzpicture}};
      \draw[blue,thick] (2.4,.7)--(1.7,1.5);
      \draw[blue] (3,1.8) node {\begin{tikzpicture}[scale=.4]
        \draw (0,0) node {$\bullet$};
        \draw (1,0) node {$\bullet$};
        \draw (2,0) node {$\bullet$};
        \draw (3,0) node {$\bullet$};
        \draw (0,0)--(3,0);
        \draw (0,0)--(-.5,.5);
        \draw (0,0)--(-.5,-.5);
        \draw (1,0)--(1,0.7);
        \draw (2,0)--(2,0.7);
        \draw (3,0)--(3.5,.5);
        \draw (3,0)--(3.5,-.5);
        \draw (-.5,.5) node[left] {1};
        \draw (-.5,-.5) node[left] {2};
        \draw (1,0.7) node[above] {3};
        \draw (2,0.7) node[above] {4};
        \draw (3.5,.5) node[right] {5};
        \draw (3.5,-.5) node[right] {6};
      \end{tikzpicture}};
    \draw[blue,thick] (2.4,1.7)--(1.5,1.8);
      \draw[blue] (-1,.7) node {\begin{tikzpicture}[scale=.4]
        \draw (0,0) node {$\bullet$};
        \draw (1,0) node {$\bullet$};
        \draw (2,0) node {$\bullet$};
        \draw (3,0) node {$\bullet$};
        \draw (0,0)--(3,0);
        \draw (0,0)--(-.5,.5);
        \draw (0,0)--(-.5,-.5);
        \draw (1,0)--(1,0.7);
        \draw (2,0)--(2,0.7);
        \draw (3,0)--(3.5,.5);
        \draw (3,0)--(3.5,-.5);
        \draw (-.5,.5) node[left] {1};
        \draw (-.5,-.5) node[left] {2};
        \draw (1,0.7) node[above] {6};
        \draw (2,0.7) node[above] {5};
        \draw (3.5,.5) node[right] {3};
        \draw (3.5,-.5) node[right] {4};
      \end{tikzpicture}};
    \draw[blue,thick] (-.4,.7)--(.3,1.5);
      \draw[blue] (-1,1.8) node {\begin{tikzpicture}[scale=.4]
        \draw (0,0) node {$\bullet$};
        \draw (1,0) node {$\bullet$};
        \draw (2,0) node {$\bullet$};
        \draw (3,0) node {$\bullet$};
        \draw (0,0)--(3,0);
        \draw (0,0)--(-.5,.5);
        \draw (0,0)--(-.5,-.5);
        \draw (1,0)--(1,0.7);
        \draw (2,0)--(2,0.7);
        \draw (3,0)--(3.5,.5);
        \draw (3,0)--(3.5,-.5);
        \draw (-.5,.5) node[left] {1};
        \draw (-.5,-.5) node[left] {2};
        \draw (1,0.7) node[above] {5};
        \draw (2,0.7) node[above] {6};
        \draw (3.5,.5) node[right] {3};
        \draw (3.5,-.5) node[right] {4};
      \end{tikzpicture}};
    \draw[blue,thick] (-.4,1.7)--(.5,1.8);
      \draw[blue] (2,3.5) node {\begin{tikzpicture}[scale=.4]
        \draw (0,0) node {$\bullet$};
        \draw (1,0) node {$\bullet$};
        \draw (2,0) node {$\bullet$};
        \draw (3,0) node {$\bullet$};
        \draw (0,0)--(3,0);
        \draw (0,0)--(-.5,.5);
        \draw (0,0)--(-.5,-.5);
        \draw (1,0)--(1,0.7);
        \draw (2,0)--(2,0.7);
        \draw (3,0)--(3.5,.5);
        \draw (3,0)--(3.5,-.5);
        \draw (-.5,.5) node[left] {3};
        \draw (-.5,-.5) node[left] {4};
        \draw (1,0.7) node[above] {1};
        \draw (2,0.7) node[above] {2};
        \draw (3.5,.5) node[right] {5};
        \draw (3.5,-.5) node[right] {6};
      \end{tikzpicture}};
        \draw[blue,thick] (1.9,3.1)--(1.1,2.2);
      \draw[blue] (0,3.5) node {\begin{tikzpicture}[scale=.4]
        \draw (0,0) node {$\bullet$};
        \draw (1,0) node {$\bullet$};
        \draw (2,0) node {$\bullet$};
        \draw (3,0) node {$\bullet$};
        \draw (0,0)--(3,0);
        \draw (0,0)--(-.5,.5);
        \draw (0,0)--(-.5,-.5);
        \draw (1,0)--(1,0.7);
        \draw (2,0)--(2,0.7);
        \draw (3,0)--(3.5,.5);
        \draw (3,0)--(3.5,-.5);
        \draw (-.5,.5) node[left] {3};
        \draw (-.5,-.5) node[left] {4};
        \draw (1,0.7) node[above] {2};
        \draw (2,0.7) node[above] {1};
        \draw (3.5,.5) node[right] {5};
        \draw (3.5,-.5) node[right] {6};
      \end{tikzpicture}};
    \draw[blue,thick] (.1,3.1)--(.9,2.4);
    \draw[blue] (3.5,3.5) node {\begin{tikzpicture}[scale=.4]
        \draw (0,0) node {$\bullet$};
        \draw (1,0) node {$\bullet$};
        \draw (2,0) node {$\bullet$};
        \draw (1,1) node {$\bullet$};
        \draw (0,0)--(2,0);
        \draw (1,0)--(1,1);
        \draw (0,0)--(-.5,.5);
        \draw (0,0)--(-.5,-.5);
        \draw (1,1)--(.5,1.5);
        \draw (1,1)--(1.5,1.5);
        \draw (2,0)--(2.5,.5);
        \draw (2,0)--(2.5,-.5);
        \draw (-.5,.5) node[left] {1};
        \draw (-.5,-.5) node[left] {2};
        \draw (.5,1.5) node[above] {3};
        \draw (1.5,1.5) node[above] {4};
        \draw (2.5,.5) node[right] {5};
        \draw (2.5,-.5) node[right] {6};
      \end{tikzpicture}};
    \draw[blue,thick] (3,3)--(1,1.75);
    \end{tikzpicture}
    \caption{The tropicalization of $\CR_{\mathcal{T}},$ over a single
      cone of the codomain}
    \label{fig:tropicalization}
  \end{figure}
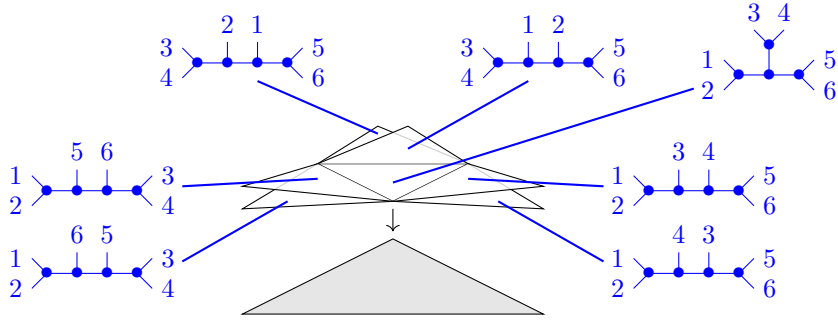
\end{ex}

\subsubsection{Goldner's algorithm}\label{sec:Goldner}
Goldner's algorithm \cite[Lem. 3.11]{Goldner2021}, which we mentioned above and
now describe, gives a recursion for $d_{\mathcal{T}}$ via tropical
geometry. Choose $e=\{v_1,v_2,v_3,v_4\}\in E$, and choose two elements
of $e$, say $v_1,v_2$ without loss of generality. Given
$V'\subseteq V$ such that $v_1,v_2\in V',$ $v_3,v_4\not\in V',$ and
for all $e'\in E\setminus\{e\}$, $\abs{V'\cap e'}\ne2,$ we may create
two new hypergraphs $\mathcal{T}',\mathcal{T}''$ as follows. The
vertices of $\mathcal{T}'$ are $V'\sqcup\{v'\}$ and the edges are
$\{e'\in E:V'\cap e'\ge3\}$, where if $v\in (V\setminus V')\cap e'$,
we replace $v$ with $v'$ in $e'$. The vertices of $\mathcal{T}''$ are
$(V\setminus V')\sqcup\{v''\}$, and the edges are
$\{e'\in E:(V\setminus V')\cap e'\ge3\}$, where if $v\in V'\cap e'$,
we replace $v$ with $v''$ in $e'$. Then by \cite[Lem.
3.11]{Goldner2021}, we have
\begin{align}\label{eq:GoldnerRecursion}
  d_{\mathcal{T}}&=\sum_{V'}d_{\mathcal{T}'}d_{\mathcal{T}''}
\end{align}
%$$d_{\mathcal{T}}=\sum_{V'}d_{\mathcal{T}'}d_{\mathcal{T}''}
%\begin{cases}
%  1&\abs{E}=1\\
  %\sum_{V'}d_{\mathcal{T}'}d_{\mathcal{T}''}%&\abs{E}>1,
  % \end{cases}$$
where $V'$ runs over subsets of $V$ as above. Note that $\mathcal{T}'$
and $\mathcal{T}''$ have strictly fewer edges than $\mathcal{T}$, as
both edge sets are naturally subsets $E\setminus\{e\}$. This implies
that any cross-ratio degree can be computed recursively from
\eqref{eq:GoldnerRecursion}, using the trivial base case
$d_{\mathcal{T}}=1$ when $\abs{V}=\{1,2,3\}$ and $E=\emptyset$ --- or
alternatively the less-trivial base case $d_{\{1,2,3,4\}}=1.$ The
algorithm can be interpreted as calculating a fiber of the map
$\CR_{\mathcal{T}}^{\trop}$ in Section \ref{sec:TropicalCoordinate},
over a carefully chosen point where all preimage points map with
multiplicity 1.

Goldner's recursion, together with other tropical computational
techniques, gives an algorithm for computing genus-zero tropical plane
curve counts with cross-ratio constraints; these counts had been the
subject of the correspondence theorem of Tyomkin \cite{Tyomkin2017}
mentioned in Section \ref{sec:TropicalCoordinate}.  Relatedly but
separately, Goldner \cite{Goldner2020Generalizing} generalized
Kontsevich's recursive algorithm \cite{KontsevichManin1994} for
rational plane curve counts to allow multiple cross-ratio constraints;
cross-ratio degrees appear in this recursion as initial values. (See
\cite{Goldner2020Thesis} for details.)

\begin{ex}
  This example also appears in \cite[Ex. 3.10]{Goldner2021}. Let
  $\mathcal{T}$ be the hypergraph of Example
  \ref{ex:CoordinateCalculation}. Let $(v_1,v_2,v_3,v_4)=(1,2,3,4).$
  There are two choices for $V',$ namely $\{1,2,5\}$ or $\{1,2,6\}.$
  Taking $V'=\{1,2,5\}$ gives new hypergraphs $\mathcal{T}'$ with
  vertices $\{1,2,5,v'\}$ and the single edge $\{1,2,5,v'\}$, and
  $\mathcal{T}''$ with vertices $\{3,4,6,v''\},$ and the single edge
  $\{3,4,6,v''\}.$ We have $d_{\mathcal{T}'}=d_{\mathcal{T}''}=1$ by
  the base case, giving a contribution of 1 to $d_{\mathcal{T}}$. The
  other choice $V'=\{1,2,6\}$ similarly contributes 1,
  i.e. $d_{\mathcal{T}}=1+1=2.$
\end{ex}

%\begin{itemize}
% \item In small cases, they may be computed directly in coordinates,
%   using either algebraic geometry (Example
%   \ref{ex:CoordinateCalculation}) or tropical geometry (Example
%   \ref{ex:tropical}).
% \item They may be realized as intersection numbers on the moduli space
%   $\Mbar_{0,n}$ of marked genus-zero stable curves, and computed via
%   well-known presentations of the cohomology ring. (See Section
%   \ref{sec:Knudsen}.)
% \item They may be realized as Gromov-Witten invariants of a toric
%   variety, so may be computed via Kontsevich's localization argument
%   or Givental's mirror theorem (Remark \ref{rem:Gromov-Witten}).
% \item Goldner \cite[Lemma 3.11]{Goldner2021} has given a beautiful
%   recursion via tropical geometry (see also \cite{Tyomkin2017}), from
%   which any $d_{\mathcal{T}}$ may be computed from the base case
%   $d_{\{\{1,2,3,4\}\}}=1$.
  % Computing any particular $d_{\mathcal{T}}$
  % requires making many choices; it is not combinatorially obvious that
  % the answer is independent of these choices.
%\end{itemize}

\subsubsection{Intersection theory on \texorpdfstring{$\Mbar_{0,n}$}{M\_\{0,n\}-bar}}\label{sec:IntersectionTheory}
A natural intersection theoretic approach to calculating
$d_{\mathcal{T}}$ is as follows. Recall that $M_{0,n}$ has a
compactification $\Mbar_{0,n}$ \cite{Knudsen1983}, the space of
$n$-marked stable rational curves, whose cohomology is relatively
well-understood. The cross-ratio degree $d_{\mathcal{T}}$ is equal to
the degree of the pullback of the class $\prod_{e\in E}H_e$ of a point
in $(\P^1)^{n-3}$ under $\CR_{\mathcal{T}}.$ That is,
$$d_{\mathcal{T}}=\deg\left(\prod_{e\in
    E}\CR_{\mathcal{T}}^*(H_e)\right).$$ At first glance, this
approach appears similar to the one we use to prove Theorem
\ref{thm:Matchings}. However, the intersection problems are not
related in a straightforward way; e.g. here there are $n-3$
subvarieties being intersected rather than $n$ subvarieties.

Recall \cite{Keel1992} that $H^*(\Mbar_{0,n},\Q)$ is generated as a
$\Q$-algebra by the classes of \emph{boundary divisors} $D_S$ with
$S\subseteq\{1,\ldots,n\}$ and $2\le\abs{S}\le n-2,$ with additive
relations as characterized in \cite{KontsevichManin1994}. We may
express $\CR_{\mathcal{T}}^*(H_e)$ in terms of boundary divisors;
e.g. if $e=\{1,2,3,4\},$ then $\CR_{\mathcal{T}}^*(H_e)$ can be
written as the sum of all boundary divisors $D_S$ with $1,2\in S$ and
$3,4\not\in S.$ One may thus express $d_{\mathcal{T}}$ as a truly
enormous sum of monomials in boundary divisors. While the class of
each monomial can be computed rather straightforwardly, the
combinatorics of working with so many terms --- which can turn out to
be negative --- appears prohibitively difficult; computing a single
cross-ratio degree this way is feasible, but gaining insight into how
$d_{\mathcal{T}}$ depends on the structure of $\mathcal{T}$ is
probably not.

\begin{ex}\label{ex:BoundaryStrataProduct}
  Let $\mathcal{T}$ be as in Example
  \ref{ex:CoordinateCalculation}. Then $d_{\mathcal{T}}$ is the degree
  of the dimension-zero intersection
  \begin{align*}
    (D_{\{1,2\}}+&D_{\{1,2,5\}}+D_{\{1,2,6\}}+D_{\{1,2,5,6\}})\\
                 &\cdot(D_{\{1,2\}}+D_{\{1,2,3\}}+D_{\{1,2,4\}}+D_{\{1,2,3,4\}})\cdot(D_{\{3,4\}}+D_{\{1,3,4\}}+D_{\{2,3,4\}}+D_{\{1,2,3,4\}}).
  \end{align*}
  The 64 terms of this sum can be evaluated by a routine calculation
  expressing products of boundary divisors in terms of \emph{boundary
    strata} and \emph{$\psi$ classes} \cite{KockNotes}; 50 terms are
  zero, 8 are equal to 1 (times the class of a point), and 6 are equal
  to $-1,$ giving $d_{\mathcal{T}}=2.$
\end{ex}

A marginally more promising approach is to use a different generating
set for $H^*(\Mbar_{0,n},\Q)$ introduced in
\cite{Singh2004,EtingofHenriquesKamnitzerRains2010}, consisting of
divisors $\Pi_S$ for $S\subseteq\{2,\ldots,n\}$ with $\abs{S}\ge3.$
This gives a simpler expression:
\begin{align*}
  \CR_{\mathcal{T}}^*(H_e)=
  \begin{cases}
    \Pi_{e\setminus\{1\}}&1\in e\\
    \left(\sum_{v\in e}\Pi_{e\setminus\{v\}}\right)-2\Pi_e&1\not\in e.
  \end{cases}
\end{align*}
Again, it is straightforward to evaluate a degree-$(n-3)$ monomial in
the divisors $\Pi_S$ via additive relations, see e.g. \cite[Lem.
5.4]{EtingofHenriquesKamnitzerRains2010}. However, the existence of
negative signs in the factors of
$\prod_{e\in E}\CR_{\mathcal{T}}^*(H_e)$ make it difficult to study
how $d_{\mathcal{T}}$ depends combinatorially on $\mathcal{T}.$

\begin{ex}
  The product in \ref{ex:BoundaryStrataProduct} is rewritten
  as
  $$\Pi_{\{2,3,4\}}\cdot\Pi_{\{2,5,6\}}\cdot(\Pi_{\{3,4,5\}}+\Pi_{\{3,4,6\}}+\Pi_{\{3,5,6\}}+\Pi_{\{4,5,6\}}-2\Pi_{\{3,4,5,6\}}).$$
  The five terms of this sum can be evaluated using \cite[Lem.
  5.4]{EtingofHenriquesKamnitzerRains2010} to give
  $d_{\mathcal{T}}=1+1+1+1-2=2.$
\end{ex}

\subsubsection{Gromov-Witten theory}\label{sec:GromovWitten} We show
(Corollary \ref{cor:GromovWitten}, which combines Theorem
\ref{thm:TurnIntoCurveCounting} with a result of
\cite{FultonPandharipande1997}) that $d_{\mathcal{T}}$ is a
\emph{Gromov-Witten invariant} of $(\P^1)^{n-3}$, a type of
curve-counting invariant known for its relationship to mirror
symmetry. % Precisely, $d_{\mathcal{T}}$ is equal to the integral, over
% the Kontsevich space $\Mbar_{0,n}((\P^1)^E,(1,\ldots,1))$, of the
% class $(\pi_I\circ\ev)^*([pt])$ in the notation of
% \eqref{eq:CommDiagGW} (on page \pageref{eq:CommDiagGW}).
As $(\P^1)^{n-3}$ is toric, $d_{\mathcal{T}}$ may be computed via
well-known techniques such as torus localization \cite{Kontsevich1995}
or mirror theorems \cite{Givental1998toric}. Again, however, it
appears difficult to find a relationship between $d_{\mathcal{T}}$ and
the combinatorial structure of $\mathcal{T}$ via these methods.

\subsubsection{Castravet-Tevelev's hypertree projections}\label{sec:Hypertrees}
Castravet-Tevelev prove \cite[Thm. 3.2]{CastravetTevelev2013} that
$d_{\mathcal{T}}\le1$ in the special case where
$\bigcap_{e\in E}e\ne\emptyset$. In this case, they show that
$d_{\mathcal{T}}=1$ if and only if $G_{\mathcal{T}}$ has \emph{surplus
  3}; see Section \ref{sec:Analysis}, and Question
\ref{question:SurplusNonzero} in particular. We note that
\cite[Thm. 3.2]{CastravetTevelev2013} can also be recovered via the
intersection product calculation in Section
\ref{sec:IntersectionTheory} using the generators $\{\Pi_S\}$; in the
case $\bigcap_{e\in E}e\ne\emptyset$, the product in question has a
single term, and \cite[Lem. 5.4]{EtingofHenriquesKamnitzerRains2010}
implies that the term is equal to 0 or 1, according to whether or not
$G_{\mathcal{T}}$ has surplus 3.

\section{Reformulation as a curve-counting
  problem}\label{sec:Reformulation} In this section we prove Theorem
\ref{thm:TurnIntoCurveCounting}, which recasts $d_{\mathcal{T}}$ as a
curve-counting invariant. Let $S$ be a set with $\abs{S}=k\ge1.$ Let
$\M_{\P^1,S}$ (or $\M_{\P^1,k}$) denote the moduli space of smooth
connected curves in $(\P^1)^S$ of multidegree $(1,\ldots,1)$, more
commonly denoted by $M_{0,0}((\P^1)^S,(1,\ldots,1)).$ Then
$\M_{\P^1,S}$ is a smooth quasiprojective variety of dimension
$3(k-1)$ \cite{FultonPandharipande1997}. A curve $C\in\M_{\P^1,S}$ is
necessarily isomorphic to $\P^1$, as $C$ is irreducible and any
composition $C\to(\P^1)^S\xrightarrow{\pi_s}{}\P^1$ has degree 1,
hence is an isomorphism.

Suppose $S'\subseteq S$ is nonempty. The projection of a
$(1,\ldots,1)$-curve along $\pi_{S'}:(\P^1)^S\to(\P^1)^{S'}$ is a
$(1,\ldots,1)$-curve, yielding a natural map
$\M_{\P^1,S}\to\M_{\P^1,S'}$.
% We will reformulate $d_{\mathcal{T}}$ as an enumeration of curves
% $C\subseteq\M_k$ satisfying incidence conditions relative to the
% following simple subvarieties, where some coordinates are constrained
% to constants:
Given $P\in(\P^1)^{S'},$ we define $Y(S',P)=\pi_{S'}^{-1}(P).$ We call
a subvariety of this form a \textbf{\textit{linear subvariety}} of
    $(\P^1)^S$ of type $S'$.
% $\{a_s\in\P^1:s\in S'\},$ we define
% $$Y(S',\{a_s\})=\{(z_s)_{s\in S}\in(\P^1)^S:\text{$z_s=a_s$ for all
%   $s\in S'$}\},$$ i.e. the preimage of $(a_s)$ under the projection
% $(\P^1)^S\to(\P^1)^{S'}.$ We refer to such a subvariety as a
% \emph{linear subvariety of $(\P^1)^S$ of type $S'$.}
For any $S'\subseteq S$, the group $\PGL(2)^S$ acts transitively on
linear subvarieties of $(\P^1)^S$ of type $S'$. Note that if
$C\in\M_{\P^1,S}$ intersects a proper linear subvariety $Y(S',P)$,
then $C\cap Y(S',P)$ is a single reduced point; this follows from the
fact that $Y(S',P)$ is contained in a codimension-1 linear subvariety,
whose intersection number with $C$ is 1. We denote by
$Z(S',P)\subseteq\M_{\P^1,S}$ the \emph{incidence subvariety}
$Z(S',P)=\{C\in\M_{\P^1,S}:C\cap Y(S',P)\ne\emptyset\}.$

Let $\mathcal{T}=(V,E)$ be a 4-uniform hypergraph with vertex set
$V=\{1,\ldots,n\}$ and edge set $E$, where $\abs{E}=n-3$. Let
$I=\{(v,e)\in V\times E:v\in e\}$ denote the incidence set, and for
$v\in V,$ let $E_v=\{e\in E:v\in e\}$.
% Fix general scalars
% $a_1,\ldots,a_{n-3}.$ For each $v\in V,$ let $E_v=\{e\in E:v\in E\}$,
% and ???????  $Z_v=\pi_{S_v}^{-1}(q_v)\subseteq(\P^1)^n$.
\begin{thm}\label{thm:TurnIntoCurveCounting}
  For each $v\in V,$ fix a general linear subvariety
  $Y_v\subseteq(\P^1)^E$ of type $E_v,$ i.e. $Y_v=Y(E_v,P_v)$ for a
  general $P_v\in(\P^1)^{E_v}.$ Then there are exactly
  $d_{\mathcal{T}}$ curves $C\in\M_{\P^1,E}$ such that
  $C\cap Y_v\ne\emptyset$ for all $v\in V$. Equivalently, the
  corresponding incidence subvarieties
  $Z_v=Z(E_v,P_v)\subseteq\M_{\P^1,E}$ intersect in exactly
  $d_{\mathcal{T}}$ points.
\end{thm}
\begin{proof}
  For $A$ a finite set, let $\M_{\P^1,S,A}$ denote the moduli space of
  smooth connected curves $C\subseteq(\P^1)^S$ of multidegree
  $(1,\ldots,1)$ together with an injection $\iota:A\into C$. This is
  a smooth quasiprojective variety of dimension
  $3(\abs{S}-1)+\abs{A},$ see \cite{FultonPandharipande1997}. There is
  a natural map $\pr_1:\M_{\P^1,S,A}\to\M_{\P^1,S}$ that forgets $A$,
  and when $\abs{A}\ge3$ there is a natural map
  $\pr_2:\M_{\P^1,S,A}\to M_{0,A}=M_{0,\abs{A}}$ that forgets the map $C\into(\P^1)^S.$

  We have a commutative diagram:
  \begin{align}\label{eq:CommDiagGW}
    \begin{tikzcd}[ampersand replacement=\&]
      \&\M_{\P^1,E,V}\arrow[dl,"\pr_1"]\arrow[d,"\pr_2"]\arrow[r,"\ev"]\&(\P^1)^{V\times
        E}\setminus\Delta_1\arrow[r,"\pi_{I}"]\&(\P^1)^{I}\setminus\Delta_2\arrow[d,"\CR^{\otimes E}"]\\
      \M_{\P^1,E}\&M_{0,\abs{V}}\arrow[rr,"\CR_{\mathcal{T}}"]\&\&(\P^1\setminus\{\infty,0,1\})^{E}\setminus\Delta_3
    \end{tikzcd}
  \end{align}
  where
  \begin{itemize}
  \item $\Delta_1\subseteq(\P^1)^{V\times E}$ is the set
    $\{(z_{(v,e)})_{(v,e)\in V\times E}:z_{(v,e)}=z_{(v',e)}\text{ for
      some $e$, $v\ne v'$}\},$
  \item $\Delta_2\subseteq(\P^1)^I$ is the set
    $\{(z_{(v,e)})_{(v,e)\in I}:z_{(v,e)}=z_{(v',e)}\text{ for some
      $e$, $v\ne v'$}\},$ and
      % large diagonal
    % $\{(z_1,\ldots,z_4):z_i=z_j\text{ for some $i\ne j$}\},$
  \item $\Delta_3\subseteq(\P^1\setminus\{\infty,0,1\})^{E}$ is the
    large diagonal
    $\{(z_e)_{e\in E}:z_{e_1}=z_{e_2}\text{ for some $e_1\ne e_2$}\}.$
  \end{itemize}
  We claim the rectangle in \eqref{eq:CommDiagGW} is Cartesian. Indeed
  commutativity gives a natural map
  $F:\M_{\P^1,E,V}\to
  M_{0,V}\times_{(\P^1\setminus\{\infty,0,1\})^{E}\setminus\Delta_3}(\P^1)^I\setminus\Delta_2.$
  We now define a map
  $$G:M_{0,V}\times_{(\P^1\setminus\{\infty,0,1\})^{E}\setminus\Delta_3}(\P^1)^I\setminus\Delta_2\to\M_{\P^1,E,V}$$ that commutes with $\pr_2$ and
  $\pi_I\circ\ev,$ and is thus guaranteed to be an inverse of $F$, as
  follows. Let $(C,\iota:V\into C)\in M_{0,V}$ and
  $(z_{v,e})_{(v,e)\in I}\in(\P^1)^{I}\setminus\Delta_2$ be points
  (over a base scheme $S$) such that for all
  $e=(v_1,v_2,v_3,v_4)\in E$ with $v_1<v_2<v_3<v_4,$
  $$\CR(z_{(v_1,e)},\ldots,z_{(v_4,e)})=\CR(\iota(v_1),\ldots,\iota(v_4)).$$
  % $(z_{(e,1)},\ldots,z_{(e,4)})_{e\in
  %   E})\in((\P^1)^{4}\setminus\Delta_2)^{n-3}$ such that for all
  % $e=(i_1,i_2,i_3,i_4)\in E,$ with $i_1<i_2<i_3<i_4,$
  % $\CR(z_{(e,1)},\ldots,z_{(e,4)})=\CR(p_{i_1},\ldots,p_{i_4}).$
  We define $G((C,\iota),(z_{(v,e)})_{(v,e)\in I})$ to be the marked
  curve $(C,\iota),$ together with the embedding
  $f=(f_e)_{e\in E}:C\into(\P^1)^E$ uniquely determined by
  $$(f_e(\iota(v_1)),f_e(\iota(v_2)),f_e(\iota(v_3)))=(z_{(v_1,e)},z_{(v_2,e)},z_{(v_3,e)}).$$
  The fact that $G$ commutes with $\pr_2$ is automatic, and the fact
  that $G$ commutes with $\pi_I\circ\ev$ follows from
  $f_e(v_{i_4})=z_{(v_4,e)},$ which is guaranteed by the assumption
  $\CR(z_{(v_1,e)},\ldots,z_{(v_4,e)})=\CR(\iota(v_1),\ldots,\iota(v_4)).$
  Thus \eqref{eq:CommDiagGW} is Cartesian. Thus $\pi_I\circ\ev$ is generically finite of degree
  $d_{\mathcal{T}}$ by surjectivity of $\CR^{\otimes E}$.

  The set of curves $C\in\M_{\P^1,E}$ that intersect $Y_v$ for all
  $v\in V$ is precisely $\pr_1((\pi_I\circ\ev)^{-1}((P_v)_{v\in V})).$
  On the other hand, given
  $C\in\pr_1((\pi_I\circ\ev)^{-1}((P_v)_{v\in V}))$, as $Y_v$ is a
  linear subvariety of $(\P^1)^{E},$ $C$ intersects $Y_v$ in a single
  reduced point. Thus $C$ determines a unique point of
  $\M_{\P^1,E,V}$, so there are exactly $d_{\mathcal{T}}$ such curves.
\end{proof}
\begin{cor}\label{cor:GromovWitten}
  The cross-ratio degree $d_{\mathcal{T}}$ coincides with the
  Gromov-Witten invariant
  $$\left\langle\{Y_v\}_{v\in
    V}^{\textcolor{white}{I}}\right\rangle_{0,n,(1,\ldots,1)}^{(\P^1)^E}=\int_{\Mbar_{0,n}((\P^1)^E,(1,\ldots,1))}(\pi_I\circ\ev)^*([pt]).$$
\end{cor}
\begin{proof}
  This is immediate from Theorem \ref{thm:TurnIntoCurveCounting} and
  \cite[Lem. 14]{FultonPandharipande1997}.
\end{proof}
% \begin{remark}\label{rem:Gromov-Witten}
%   By \cite[Lemma 14]{FultonPandharipande1997}, $d_{\mathcal{T}}$ is
%   also equal to a Gromov-Witten invariant of $(\P^1)^{E}$, namely the
%   integral of $(\pi_I\circ\ev)^*([pt])$ over the compactification
%   $\Mbar_{0,n}((\P^1)^{E},(1,\ldots,1))$ of $\M_{\P^1,E,V}$. As such,
%   $d_{\mathcal{T}}$ can be computed via standard methods in
%   Gromov-Witten theory such as torus localization. It appears
%   difficult to do this calculation in a way that makes the answer
%   expressible in terms of the combinatorics of $\mathcal{T}.$
% \end{remark}

\section{Proof of Theorem \ref{thm:Matchings}}\label{sec:Proof}

By Corollary \ref{cor:GromovWitten}, $d_{\mathcal{T}}$ can be
expressed as an intersection number on the Kontsevich compactification
$\Mbar_{0,n}((\P^1)^{n-3},(1,\ldots,1))$ of $\M_{\P^1,E,V}$. We will
prove Theorem \ref{thm:Matchings} by computing an intersection number
in a similar spirit, but rather on a very simple compactification of
$\M_{\P^1,E}$ whose cohomology groups are much smaller than those of
$\Mbar_{0,n}((\P^1)^{n-3},(1,\ldots,1))$.

\subsection{Cohomology classes of incidence subvarieties}\label{sec:Cohomology} For
$e\in E,$ we define an isomorphism
$\rho_e:\M_{\P^1,E}\to\PGL(2)^{E\setminus\{e\}}$ as follows. Fix
$C\in\M_{\P^1,E}.$ The projection onto the $e$-th coordinate of
$(\P^1)^{E}$ defines an isomorphism $\nu_e:C\cong\P^1$. For $e'\in E$,
$\nu_{e'}\circ\nu_e^{-1}:\P^1\to\P^1$ is an automorphism,
i.e. $\nu_{e'}\circ\nu_e^{-1}\in\PGL(2).$ Let
$$\rho_e(C)=(\nu_{e'}\circ\nu_e^{-1})_{e'\in
  E\setminus\{e\}}\in\PGL(2)^{E\setminus\{e\}}.$$ The map $\rho_e$ has
a clear inverse; given
$(A_{e'})_{e'\in E\setminus\{e\}}\in\PGL(2)^{E\setminus\{e\}},$ one
embeds $\P^1$ into $(\P^1)^{n-3}$ via the identity map in the $e$-th
coordinate, and via $A_{e'}$ in the $e'$-th coordinate for $e'\ne e.$ We have an open immersion $\PGL(2)\into\P^3,$ hence for each $e\in E,$
we obtain a compactification
$\bar{\rho_e}:\M_{\P^1,E}\into(\P^3)^{E\setminus\{e\}}$. We now
calculate the class of the Zariski closure of each incidence
subvariety $Z_v$ in the cohomology of $(\P^3)^{E\setminus\{e\}}.$
\begin{prop}\label{prop:CalculateClass}
  Fix $v\in V$ and $e\in E.$ For $P\in(\P^1)^{E_v},$ let
  $Z_v=Z(E_v,P)\subseteq\M_{\P^1,E}$ be the incidence subvariety as in
  Section \ref{sec:Reformulation}. Let $\bar{Z_{v}^e}$ denote the Zariski
  closure of $Z_v$ under the embedding $\bar{\rho_e}$. For
    $e'\in E\setminus\{e\},$ let $H_{e'}$ denote the pullback of the
    hyperplane class on $\P^3$ along the $e'$-th projection. Then
  \begin{align}
    [\bar{Z_v^e}]=
    \begin{cases}
      \prod_{e'\in E_v\setminus\{e\}}H_{e'}&v\in e\\
      \mathbf{e}_{\abs{E_v}-1}(\{H_{e'}:e'\in E_v\})&v\not\in e,
    \end{cases}
  \end{align}
  where $\mathbf{e}_{\abs{E_v}-1}$ is the elementary symmetric polynomial of
  degree $\abs{E_v}-1.$
\end{prop}
\begin{proof}
  Note that the natural $\PGL(2)^E$-action on $(\P^1)^E$ induces an
  action on $\M_{\P^1,E}$, and that this action extends to
  $(\P^3)^{E\setminus\{e\}}.$ (Explicitly, $(A_{e'})_{e'\in E}$ sends
  $(B_{e''})_{e''\in E\setminus\{e\}}$ to
  $(A_{e''}B_{e''}A_e^{-1})_{e''\in E\setminus\{e\}}.$) As observed
  above, the action is transitive on linear subvarieties, and hence we
  may assume without loss of generality that
  $P=P_0=(0,\ldots,0)\in(\P^1)^{E_v}.$

  If $v\in e,$ then
  $$\rho_e(Z(E_v,P_0))=\{(B_{e'})_{e'\in
    E\setminus\{e\}}:\text{$B_{e'}\cdot[1:0]=[1:0]$ if $v\in e'$}\};$$
  that is, the lower left entry of the matrix $B_{e'}$ is zero for
  $e'\ni v.$ The same condition holds after taking the Zariski
  closure, so $[\bar{Z_v^e}]$ is the product of hyperplane classes
  pulled back from $(\P^3)^{E_v\setminus\{e\}},$ as claimed.

  Suppose $v\not\in e.$ % By definition, $Z_v$ is a preimage under
  % $\tau_{E_v},$ hence $[\bar{Z_v^e}]$ is pulled back from
  % $(\P^3)^{E_v}.$
  Fix $e''\in E_v.$ Consider $\bar{Z_v^e}\cap J_{\Id},$ where
  $J_A=\{(B_{e'})_{e'\in
    E\setminus\{e\}}:B_{e''}=A\}\subseteq(\P^3)^{E\setminus\{e\}}$. Note
  that $J_{\Id}\cong(\P^3)^{E\setminus\{e,e''\}},$ and
  $[J_{\Id}]=H_{e''}^3.$ Inside $J_{\Id}$, we have a natural
  identification of $\bar{Z_v^e}\cap J_{\Id}$ with
  $\rho_{e''}(Z_{E\setminus\{e\}}(E_v,P)),$ where
  $Z_{E\setminus\{e\}}(E_v,P)\subseteq\M_{\P^1,E\setminus\{e\}}$ is
  the incidence subvariety with ambient space
  $(\P^1)^{E\setminus\{e\}}$. By the previous computation, inside
  $(\P^3)^{E\setminus\{e,e''\}},$ we have
  $[\bar{Z_v^e}\cap J_{\Id}]=\prod_{e'\in E_v\setminus\{e''\}}H_{e'}.$
  We claim the intersection $\bar{Z_v^e}\cap J_{\Id}$ is transverse,
  so that in $(\P^3)^{E\setminus\{e\}}$ we have
  $[\bar{Z_v^e}\cap J_{\Id}]=H_{e''}^3\cdot\prod_{e'\in
    E_v\setminus\{e''\}}H_{e'}.$
  
  We must show that at any point of $\bar{Z_v^e}\cap J_{\Id}$, the
  tangent space to $\bar{Z_v^e}$ spans the normal space to the fiber
  $J_{\Id}$, which is naturally identified with the tangent space to
  $\P^3$ at $\Id.$ To see this, note that since $v\not\in e,$
  $\bar{Z_v^e}$ is invariant under the action of $\PGL(2)$ on
  $(\P^3)^{E\setminus\{e\}}$ induced by the action on the $e$-th
  coordinate of $(\P^1)^e.$ This action respects the projection from
  $(\P^3)^{E\setminus\{e\}}$ to the $e''$-th coordinate $\P^3$, and
  its derivative acts transitively on the tangent space to $\P^3$ at
  $\Id$; both follow from the explicit form of the action at the
  beginning of this proof. Thus at any point of
  $\bar{Z_v^e}\cap J_{\Id}$, the tangent space of $\bar{Z_v^e}$ spans
  the normal space to $J_{\Id},$ i.e. the intersection is transverse.

  The only degree-$(\abs{E_v}-1)$ element
  $$\alpha\in H^*((\P^3)^{E\setminus\{e\}})\cong\Z[\{H_{e'}\}_{e'\in
    E\setminus\{e\}}]/\langle\{H_{e'}^4\}_{e'\in
    E\setminus\{e\}}\rangle$$ satisfying
  $H_{e''}^3\alpha=H_{e''}^3\cdot\prod_{e'\in
    E_v\setminus\{e''\}}H_{e'}$ for all $e''\in E_v$ is
  $\alpha=\mathbf{e}_{\abs{E_v}-1}(\{H_{e'}:e'\in E_v\})$, concluding
  the proof.
  % If $v\not\in e,$ we use Poincar\'e duality to calculate
  % $[\bar{Z_v^e}].$ By definition, $Z_v$ is a preimage under
  % $\tau_{E_v},$ hence $[\bar{Z_v^e}]$ is pulled back from
  % $(\P^3)^{E_v},$ so can be expressed as a sum of monomials in
  % $\{H_{e'}\}_{e'\in E_v}.$ Fix such a monomial
  % $m=c_m\prod_{e'\in E_v}H_{e'}^{r_{e'}}$. As
  % $\codim(\bar{Z_v^e})=\abs{E_v}-1$, we have $r_{e''}=0$ for some
  % $e''\in E_v.$ Then $
  % A monomial
  % $\prod_{e'\in E_v}H_{e'}^{r_{e'}}$ is Poincar\'e dual to the
  % monomial $\prod_{e'\in E_v}H_{e'}^{3-r_{e'}}$. We compute the intersection
  % product
  % \begin{align}\label{eq:IntersectionProduct}
  %   [\bar{Z_v^e}]\cdot H_{e''}^3\cdot\prod_{e'\in
  %   E_v\setminus\{e''\}}H_{e'}^{}.
  % \end{align}
  % If we choose the identity matrix as a representative for
  % $H_{e''}^3,$ we may compute this product in
  % $(\P^3)^{E_v\setminus\{e''\}},$ and furthermore, this product is
  % identified with $\rho_{e''}(\M_{\P^1,E_v\setminus\{e''\}}).$ Hence
  % by the previous computation, the product
  % \eqref{eq:IntersectionProduct} now reads
  % $$\left(\prod_{e'\in E_v\setminus e''}H_{e'}\right)\cdot\prod_{e'\in
  %   E_v\setminus\{e''\}}H_{e'}^2=1.$$
\end{proof}

% \begin{rem}
%   The proof of Proposition \ref{prop:CalculateClass} has the notable
%   consequence:
%   \begin{align*}
%     \bar{Z_v^e}\cong
%     \begin{cases}
%       (\P^2)^{E_v\setminus\{e\}}&v\in e\\
%       &v\not\in e\\
%     \end{cases}
%   \end{align*}

% \end{rem}

% \begin{prop}
%   Let $S\subseteq\{1,\ldots,n-3\}$ with $1\not\in S.$ Then
%   $\alpha_S=e_{\abs{S}-1}(\{H_i\}_{i\in S})$, where $e_{\abs{S}-1}$ is
%   the degree-$(\abs{S}-1)$ elementary symmetric polynomial in
%   $\{H_i\}_{i\in S}.$
% \end{prop}
% \begin{proof}
%   Use Poincar\'e duality; intersect with
%   e.g. $H_{i_1}^3H_{i_2}^2\cdots.$ The intersection with $H_1^3$ will
%   specify which point of $\P^1$ maps to the chosen points, and the
%   others then rigidify that copy of $\P^3$, giving exactly one
%   orbit. This shows that for a cocardinality-1 subset $S'\subseteq
%   S,$ the coefficient of $\prod_{i\in S'}H_i$ is 1.

%   Must also check that we don't get other stuff, i.e. if we have
%   $H_{i_1}^3H_{i_2}^3\cdots.$ we get zero. (Note pulled back from
%   smaller product of $\P^3$s.)
% \end{proof}

\subsection{Proof of the upper bound}

We will need the following, which is also one of the basic
observations about the numbers $d_{\mathcal{T}}$:
\begin{lem}\label{lem:AddEdge}
  Let $\mathcal{T}=(V,E)$ be a 4-uniform hypergraph with
  $\abs{E}=\abs{V}-3$, and fix $v_1,v_2,v_3\in V.$ Let
  $\mathcal{T}'=(V',E')$ with $V'=V\cup\{n+1\}$ and $E'=E\cup\{e'\},$
  where $e'=\{v_1,v_2,v_3,n+1\}.$ Then
  $d_{\mathcal{T}}=d_{\mathcal{T}'}.$
\end{lem}
\begin{proof}
  Consider the diagram
  \begin{align*}
    \begin{tikzcd}[ampersand replacement=\&]
      M_{0,V'}\arrow[r,"\CR_{\mathcal{T}'}"]\arrow[d,"\mu_{n+1}"]\&(\P^1)^{E'}\arrow[d,"\pi_{E}"]\\
      M_{0,V}\arrow[r,"\CR_{\mathcal{T}}"]\&(\P^1)^E
    \end{tikzcd}
  \end{align*}
  where $\mu_{n+1}$ is the map that forgets the marked point
  $\iota(n+1).$ This defines a natural map
  $F:M_{0,V'}\to M_{0,V}\times_{(\P^1)^E}(\P^1)^{E'}.$ We define a
  rational inverse
  $G:M_{0,V}\times_{(\P^1)^E}(\P^1)^{E'}\dashrightarrow M_{0,V'}$ by
  $G((C,\iota),P)=(C,\iota'),$ where $\iota':V'\to C$ is defined by
  $\iota'(v)=\iota(v)$ for $v\in V$ and $\iota'(n+1)$ is determined by
  $\CR(\iota(v_1),\iota(v_2),\iota(v_3),\iota'(n+1))=\pi_{{\{e'\}}}(P)\in\P^1.$
  Then $G$ is well-defined whenever $\pi_{{\{e'\}}}(P)$ avoids the
  finite set
  $$\{\CR(\iota(v_1),\iota(v_2),\iota(v_3),\iota(v))\}_{v\in
    V}\subseteq\P^1,$$ hence is a rational map. Where defined, $G$
  commutes with $\mu_{n+1}$ by definition, and commutes with
  $\CR_{\mathcal{T}'}$ by the condition
  $\CR_{\mathcal{T}'}(C,\iota)=\pi_{E}(P).$ Thus $G$ is a birational
  inverse to $F$. As $\CR_{\mathcal{T}}$ is generically finite of
  degree $d_{\mathcal{T}}$, so is $\CR_{\mathcal{T}'}\circ G$, hence
  so is $\CR_{\mathcal{T}'}\circ G\circ F=\CR_{\mathcal{T}'}.$ That
  is, $d_{\mathcal{T}}=d_{\mathcal{T}'}$ as claimed.
\end{proof}

\begin{proof}[Proof of Theorem \ref{thm:Matchings}]

  Let $\mathcal{T},$ $V$, $E$, and $I$ be as before, and fix
  $v_1,v_2,v_3\in V.$ By Lemma \ref{lem:AddEdge}, we have
  $d_{\mathcal{T}}=d_{\mathcal{T}'},$ where $\mathcal{T}'$ has vertex
  and hyperedge sets $V'=V\cup\{n+1\}$ and $E'=E\cup\{e'\}$ (and
  incidence set $I'$), with $e'=\{v_1,v_2,v_3,n+1\}$. For each
  $v\in V',$ fix a general linear subvariety $Y_v\subseteq(\P^1)^{E'}$
  of type $E_v$, let $Z_v\subseteq\M_{\P^1,E'}$ be the associated
  incidence variety, and let $\bar{Z_v^{e'}}$ denote the Zariski
  closure of the image of $Z_v$ under the the embedding
  $\bar{\rho_{n+1}}:\M_{\P^1,E'}\into(\P^3)^{E'\setminus\{e'\}}=(\P^3)^E.$
  \begin{claim}\label{claim:UpperBound}
    We have $d_{\mathcal{T}}\le\prod_{v\in V'}[\bar{Z_v^{e'}}].$
  \end{claim}
  \begin{proof}[Proof of Claim \ref{claim:UpperBound}]
    Let $G=\PGL(4)^E,$ let
    $$\bar{\mathcal{Z}_v^{e'}}:=\{(g,g\cdot C)\in G\times(\P^3)^E:g\in\PGL(4),C\in\bar{Z_v^{e'}}\}$$ denote the
    universal translate of $\bar{Z_v^{e'}}$, let
    $\mathcal{Z}=\bigcap_{v\in V'}\bar{\mathcal{Z}_v^{e'}}$, and let
    $\epsilon:\mathcal{Z}\to G$ denote the projection. As $G$
    acts transitively on $(\P^3)^E$, Kleiman's Bertini theorem
    \cite{Kleiman1974} implies $\epsilon$ is generically finite of
    degree $\prod_{v\in V'}[\bar{Z_v^{e'}}].$ % Let
    % $\mathcal{U}\subseteq\mathcal{Z}$ denote the subscheme consisting of
    % points $(z,g)\in\mathcal{Z}$ such that the local dimension of
    % $\epsilon^{-1}(g)$ at $(z,g)$ is zero. By Chevalley's Theorem
    % \cite[Theorem 13.1.3]{EGAIVpart3}, $\mathcal{U}$ is open.
    
    Let $\mathcal{Z}'\subseteq\mathcal{Z}$ denote the Zariski closure of
    $\epsilon^{-1}(\eta),$ where $\eta\in G$ is the generic point. Thus
    $\mathcal{Z}'$ consists of all components of
    $\mathcal{Z}$ that map dominantly to $G.$ Note that
    $\epsilon|_{\mathcal{Z}'}$ is generically finite of degree
    $\prod_{v\in V'}[\bar{Z_v^{e'}}].$ By \cite[Prop.
    15.5.3]{EGAIVpart3}, for any $g\in G,$
    $\epsilon^{-1}(g)\cap\mathcal{Z}'$ has at most
    $\prod_{v\in V'}[\bar{Z_v^{e'}}]$ connected components.
    
    Note that the fiber $\epsilon^{-1}((\Id)_{e\in E})$ is precisely
    the intersection $\bigcap_{v\in V'}\bar{Z_v^{e'}}$. In particular,
    Theorem \ref{thm:TurnIntoCurveCounting} and genericity of the
    linear subvarieties $Y_v$ imply that
    $\M_{\P^1,E'}\cap\epsilon^{-1}((\Id)_{e\in E})$ is a finite set of
    exactly $d_{\mathcal{T}}$ reduced points
    $Q_1,\ldots,Q_{d_{\mathcal{T}}}.$ Thus if we show
    $Q_1,\ldots,Q_{d_{\mathcal{T}}}\in\mathcal{Z}'$, it will imply the
    Claim \ref{claim:UpperBound}.
    % That is,
    % $\M_{\P^1,E'}\cap\epsilon^{-1}((\Id)_{e\in E})\subseteq\mathcal{U}.$
    
    % By Zariski's Main Theorem (as formulated by Grothendieck, see
    % \cite[Theorem 8.12.6]{EGAIVpart3}), there exists a factorization
    % $\mathcal{U}\into\mathcal{U}'\xrightarrow{\epsilon'}{}\PGL(4)^E$
    % such that $\epsilon'$ is finite.
    
    To see that $Q_1,\ldots,Q_{d_{\mathcal{T}}}\subseteq\mathcal{Z}'$,
    note that $\mathcal{Z}$ is the intersection of subvarieties of total
    codimension
    $$\sum_{v\in
      V'}(\abs{E'_v}-1)=\abs{I'}-\abs{V'}=3\abs{E}=\dim((\P^3)^E).$$
    Thus every irreducible component of $\mathcal{Z}$ has dimension at
    least $\dim(G),$ hence either maps dominantly to $G$, or maps to
    $G$ with positive relative dimension. If $\mathcal{Z}_0$ is a
    component of the latter type, upper semicontinuity of fiber
    dimension \cite[Thm. 13.1.3]{EGAIVpart3} implies that every
    fiber of $\epsilon|_{\mathcal{Z}_0}$ has positive dimension
    (locally at any point). We therefore must have
    $Q_i\not\in\mathcal{Z}_0$ for all $i$, as $Q_i$ is an isolated
    point of its fiber. We conclude
    $Q_1,\ldots,Q_{d_{\mathcal{T}}}\in\mathcal{Z}',$ and Claim
    \ref{claim:UpperBound} follows.
  \end{proof}

  Theorem \ref{thm:Matchings} now follows from:
  \begin{claim}\label{claim:Matchings}
    We have $\prod_{v\in V'}[\bar{Z_v^{e'}}]=P(G_{\mathcal{T}}-\{v_1,v_2,v_3\}).$
  \end{claim}
  \begin{proof}[Proof of Claim \ref{claim:Matchings}] We compute
    $\prod_{v\in V'}[\bar{Z_v^{e'}}]$ using Proposition
    \ref{prop:CalculateClass}, in the ring
    $\Z[\{H_e\}_{e\in E}]/\langle\{H_e^4\}_{e\in E}\rangle.$ As
    $\prod_{v\in V'}[\bar{Z_v^{e'}}]$ has codimension
    $\sum_{v\in V'}(\abs{E'_v}-1)=\abs{I'}-\abs{V'}=3\abs{E},$ and
    every monomial in the expansion of
    $\prod_{v\in V'}[\bar{Z_v^{e'}}]$ appears with coefficient 1 by
    Proposition \ref{prop:CalculateClass}, we must simply count
    occurences of $\prod_{e\in E}H_e^3$.
    
    The monomials in the expansion of
    $\prod_{v\in V'}[\bar{Z_v^{e'}}]$ are indexed by choices of a term
    of $\mathbf{e}_{\abs{E_v}-1}(\{H_e:e\in E_v'=E_v\})$ for each
    $v\in V'\setminus e'=V\setminus\{v_1,v_2,v_3\}$. As a term of
    $\mathbf{e}_{\abs{E_v}-1}(\{H_e:e\in E_v'=E_v\})$ contains all but
    one hyperplane class (and is therefore determined by the missing
    hyperplane class), there is a bijection $\phi$ from the set of
    terms in the expansion of $\prod_{v\in V'}[\bar{Z_v^{e'}}]$ to
    $\prod_{v\in V\setminus\{v_1,v_2,v_3\}}E_v.$ As each $e\in E$ has
    cardinality 4, such a monomial $m$ is equal to
    $\prod_{e\in E}H_e^3$ if and only if $\phi(m)$ contains each
    $e\in E$ exactly once, i.e. if and only if $\phi(m)$ determines a
    perfect matching $V\setminus\{v_1,v_2,v_3\}\to E.$
  \end{proof}
  % Have a product of $n$ terms. We want to count instances of
  % $H_1^3H_2^3\cdots H_{n-4}^3$. For each of $i=1,\ldots,n-4,$ we need
  % to pick three (distinct) of the $n$ terms where those factors come
  % from. Constraints: (1) Must pick any terms $j$ where $j$ appears in
  % row 1. (2) In each column, all but one row must get picked.
  % Must pick one of the four terms where $i$ appears, so this is
  % equivalent to \emph{not} picking 1 of them. The constraints become:
  % (1) Cannot (not) pick a term $j$ where $j$ appears in row 1, and (2)
  % In each column, exactly one row is (not) picked.
  % This is the same as choosing a perfect matching, after deleting
  % vertices in row 1.
  This completes the proof of Theorem \ref{thm:Matchings}.
\end{proof}

\begin{ex}\label{ex:matchings}
  As in Example \ref{ex:CoordinateCalculation}, let $\mathcal{T}$ be
  the hypergraph with $V=\{1,\ldots,6\}$ and $E=\{e_1,e_2,e_3\},$
  where $e_1=\{1,2,3,4\},$ $e_2=\{1,2,5,6\},$ and $e_3=\{3,4,5,6\}.$
  Deleting the vertices $\{1,2,3\}$ yields the hypergraph with
  incidence graph shown on the left in Figure \ref{fig:matchings},
  with its two matchings shown on the right.
  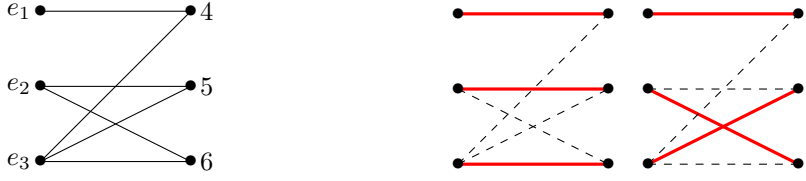
\begin{figure}
    \centering
    \begin{tikzpicture}
      \draw (0,0) node {$\bullet$} node[left] {$e_1$};
      \draw (0,-1) node {$\bullet$} node[left] {$e_2$};
      \draw (0,-2) node {$\bullet$} node[left] {$e_3$};
      \draw (2,0) node {$\bullet$} node[right] {4};
      \draw (2,-1) node {$\bullet$} node[right] {5};
      \draw (2,-2) node {$\bullet$} node[right] {6};
      \draw (0,0)--(2,0);
      \draw (0,-1)--(2,-1);
      \draw (0,-1)--(2,-2);
      \draw (0,-2)--(2,0);
      \draw (0,-2)--(2,-1);
      \draw (0,-2)--(2,-2);
    \end{tikzpicture}
    \quad\quad\quad\quad\quad\quad\quad\quad
    \begin{tikzpicture}
      \draw[very thick,red] (0,0)--(2,0);
      \draw[very thick,red] (0,-1)--(2,-1);
      \draw[dashed] (0,-1)--(2,-2);
      \draw[dashed] (0,-2)--(2,0);
      \draw[dashed] (0,-2)--(2,-1);
      \draw[very thick,red] (0,-2)--(2,-2);
      \draw (0,0) node {$\bullet$};
      \draw (0,-1) node {$\bullet$};
      \draw (0,-2) node {$\bullet$};
      \draw (2,0) node {$\bullet$};
      \draw (2,-1) node {$\bullet$};
      \draw (2,-2) node {$\bullet$};
    \end{tikzpicture}
    \begin{tikzpicture}
      \draw[very thick,red] (0,0)--(2,0);
      \draw[dashed] (0,-1)--(2,-1);
      \draw[very thick,red] (0,-1)--(2,-2);
      \draw[dashed] (0,-2)--(2,0);
      \draw[very thick,red] (0,-2)--(2,-1);
      \draw[dashed] (0,-2)--(2,-2);
      \draw (0,0) node {$\bullet$};
      \draw (0,-1) node {$\bullet$};
      \draw (0,-2) node {$\bullet$};
      \draw (2,0) node {$\bullet$};
      \draw (2,-1) node {$\bullet$};
      \draw (2,-2) node {$\bullet$};
    \end{tikzpicture}
    \caption{The incidence graph of Example \ref{ex:matchings} and its
      perfect matchings}
    \label{fig:matchings}
  \end{figure}
  Explicitly, these two matchings arise from the proof of Theorem
  \ref{thm:Matchings} as follows. Let $V'=V\cup\{7\}$ and
  $E'=E\cup\{e'\},$ where $e'=\{1,2,3,7\}$. We are intersecting incidence subvarieties
  $\bar{Z_1^{e'}},\ldots,\bar{Z_7^{e'}}\subseteq(\P^3)^3=(\P^3)^E$,
  with classes
  \begin{align*}
    [\bar{Z_1^{e'}}]=[\bar{Z_2^{e'}}]&=H_{e_1}H_{e_2}&
    [\bar{Z_3^{e'}}]&=H_{e_1}H_{e_3}&
    [\bar{Z_4^{e'}}]&=H_{e_1}+H_{e_3}\\
    [\bar{Z_5^{e'}}]=[\bar{Z_6^{e'}}]&=H_{e_2}+H_{e_3}&
    [\bar{Z_7^{e'}}]&=1.
  \end{align*}
  The product is
  \begin{align*}
  H_{e_1}^3H_{e_2}^2H_{e_3}&(H_{e_1}+H_{e_3})(H_{e_2}+H_{e_3})^2\\&=H_{e_1}^4 H_{e_2}^4 H_{e_3} + 2 H_{e_1}^4 H_{e_2}^3 H_{e_3}^2 + H_{e_1}^3 H_{e_2}^4 H_{e_3}^2 + H_{e_1}^4 H_{e_2}^2 H_{e_3}^3 + 
  2 H_{e_1}^3 H_{e_2}^3 H_{e_3}^3 + H_{e_1}^3 H_{e_2}^2 H_{e_3}^4.
  \end{align*}
  The monomial $H_{e_1}^3H_{e_2}^3H_{e_3}^3$ appears with coefficient
  2, so $d_{\mathcal{T}}\le2$ as we already knew. The left perfect
  matching
  $$(\phi^{-1}(4),\phi^{-1}(5),\phi^{-1}(6))=(e_1,e_2,e_3)$$ in Figure
  \ref{fig:matchings} corresponds to choosing the product of the terms
  $H_{e_3}$ in $[\bar{Z_4^{e'}}]$, $H_{e_3}$ in $[\bar{Z_5^{e'}}]$,
  and $H_{e_2}$ in $[\bar{Z_6^{e'}}]$. The other matching corresponds
  to choosing the product of the terms $H_{e_3}$ in
  $[\bar{Z_4^{e'}}]$, $H_{e_2}$ in $[\bar{Z_5^{e'}}]$, and $H_{e_3}$
  in $[\bar{Z_6^{e'}}]$.

  In fact, in this example the same bound arises from any choice of
  $v_1,v_2,v_3$.
\end{ex}

\begin{ex}\label{ex:NonzeroDiscrepancy}
  It is illuminating to see how the intersection calculation above can
  fail to give the exact answer for $d_{\mathcal{T}}$. This may happen
  even in small examples; we now describe the simplest case where it
  does happen. Let $\mathcal{T}$ be the hypergraph with
  $V=\{1,\ldots,6\}$ and $E=\{\{1,2,3,4\},\{1,2,3,4\}\}.$ We certainly
  have $d_{\mathcal{T}}=0$ as the map
  $\CR_{\mathcal{T}}:M_{0,5}\to(\C^*\setminus\{0,1\})^2$ factors
  through the diagonal. On the other hand,
  $G_{\mathcal{T}}\setminus\{3,4,5\}$ is the complete bipartite graph
  $K_{2,2}$, which has two perfect matchings.

  To see what is going on, we modify the hypergraph as in the proof of
  Theorem \ref{thm:Matchings}, letting $\mathcal{T}'$ have vertices
  $V'=V\cup\{6\}$ and hyperedges $E'=E\cup\{e'\},$ where
  $e'=\{3,4,5,6\}$. Then $\M_{\P^1,E'}$ is the space of
  multidegree-$(1,1,1)$ curves in $(\P^1)^3,$ and $d_{\mathcal{T}}=0$
  is the number of such curves passing through two general points
  $P_1,P_2$ and two vertical lines $L_1,L_2$. (One may see this
  directly by considering the projection onto the first two factors,
  where the constraints give a codimension-4 condition on the
  3-dimensional moduli space $\M_{\P^1,E}$). Without loss of
  generality we may choose $P_1=(0,0,0)$ and
  $P_2=(\infty,\infty,\infty)$. Under the embedding
  $\rho_{e'}:\M_{\P^1,E'}\to(\P^3)^E,$ the curves passing through
  $P_1$ and $P_2$ map to the matrices
  \begin{align*}
    \left(
    \begin{pmatrix}
      a&0\\
      0&b
    \end{pmatrix}
         ,
         \begin{pmatrix}
           c&0\\
           0&d
         \end{pmatrix}
              \right).
  \end{align*}
  The conditions that such a curve pass through $L_1$ and $L_2$ are of
  the forms $ad+\alpha_1bc=0$ and $ad+\alpha_2bc=0,$ where
  $\alpha_1\ne\alpha_2.$ Together these imply $bc=0$ and $ad=0$,
  giving two intersection points in $(\P^3)^E\setminus\M_{\P^1,E'}$:
  \begin{align*}
    &\left(
    \begin{pmatrix}
      1&0\\
      0&0
    \end{pmatrix}
         ,
         \begin{pmatrix}
           1&0\\
           0&0
         \end{pmatrix}
              \right)&&\text{and}&\left(
    \begin{pmatrix}
      0&0\\
      0&1
    \end{pmatrix}
         ,
         \begin{pmatrix}
           0&0\\
           0&1
         \end{pmatrix}
              \right).
  \end{align*}
  These two points contribute to the intersection product, but not to
  $d_{\mathcal{T}}.$ Note that the incidence subvarieties may be
  deformed in $(\P^3)^{E}$ (e.g. via the $\PGL(4)^E$ action) so that
  the intersection lies inside $\M_{\P^1,E}$; however, the moved
  subvarieties will no longer be incidence subvarieties for this
  enumerative problem.
\end{ex}

\begin{rem}
  It would be interesting to know whether Theorem \ref{thm:Matchings}
  can be proved directly from \cite{Goldner2021}.
\end{rem}

\section{Some analysis of the upper bound}\label{sec:Analysis}
Recall that for a bipartite graph $G=(A\sqcup B,I),$ the
\emph{neighborhood} of a subset $A'\subseteq A$ is
$$\mathcal{N}(A')=\mathcal{N}_G(A')=\{b\in B:\text{$(a,b)\in I$ for some $a\in A'$}\},$$
and the \emph{surplus} \cite[Sec. 1.3]{LovaszPlummer2009} of $G$ (with
respect to $A$) is
$$\sigma(G):=\min_{\substack{A'\subseteq
    A\\A'\ne\emptyset}}(\abs{\mathcal{N}(A')}-\abs{A'}).$$ Note that
$\sigma(G_{\mathcal{T}})\le 3$ as $\mathcal{T}$ is 4-uniform. Also,
recall Hall's Theorem on matchings, which states that $G$ has a
perfect matching if and only if $\sigma(G)\ge0.$
\begin{prop}\label{prop:SurplusUpperBound}
  The upper bound
  $\min_{v_1,v_2,v_3}P(G_{\mathcal{T}}-\{v_1,v_2,v_3\})$ is nonzero if
  and only if $\sigma(G_{\mathcal{T}})=3$.
\end{prop}
% \begin{proof}
%   THIS PROOF IS FOR $d_{\mathcal{T}}$! Note that
%   $\sigma(G_{\mathcal{T}})\le3$ by considering 1-element subsets of
%   $E$, as $\mathcal{T}$ is 4-uniform. First, suppose
%   $\sigma(G_{\mathcal{T}})<3$. Then there exists a nonempty subset
%   $E'\subseteq E$ such that $\abs{\mathcal{N}(E')}<k+3.$ We have the
%   commutative diagram:
%   \begin{align*}
%     \begin{tikzcd}
%       \M_{0,n}\arrow[r]\arrow[d]&\M_{0,\abs{\mathcal{N}(E')}}\arrow[d]\\
%       (\P^1)^{n-3}\arrow[r]&(\P^1)^k
%     \end{tikzcd}
%   \end{align*}
%   The right vertical arrow is not dominant by dimension counting. The
%   bottom arrow is dominant, so it follows that the left vertical arrow
%   is not dominant. Thus $d_{\mathcal{T}}=0.$
% \end{proof}
\begin{proof}
  Suppose $\sigma(G_{\mathcal{T}})<3$. Then there exists a nonempty
  subset $E'\subseteq E$ such that $\abs{\mathcal{N}(E')}<k+3.$ Let
  $v_1,v_2,v_3\in\mathcal{N}(E')$ be distinct. Then in
  $G_{\mathcal{T}}-\{v_1,v_2,v_3\},$ we have
  $\abs{\mathcal{N}(E')}<\abs{E'},$ so by (the easy direction of)
  Hall's theorem, $P(G_{\mathcal{T}}-\{v_1,v_2,v_3\})=0$.

  Now suppose $\sigma(G_{\mathcal{T}})=3$, and fix $v_1,v_2,v_3\in V.$
  For any $E'\subseteq E,$ we have $\mathcal{N}(E')\ge\abs{E'}+3$ in
  $G_{\mathcal{T}}$, so $\mathcal{N}(E')\ge\abs{E'}$ in
  $G_{\mathcal{T}}-\{v_1,v_2,v_3\}.$ By Hall's Theorem,
  $P(G_{\mathcal{T}}-\{v_1,v_2,v_3\})>0$.
\end{proof}

% (THE FOLLOWING COULD BE TOTALLY WRONG --- ACTUALLY IT FOLLOWS
% IMMEDIATELY FROM LEMMA 1.3.7 OF LOVASZ-PLUMMER)
% \begin{prop}
%   For $k\ge0,$ a bipartite graph $G=(A\sqcup B,I)$ has surplus
%   $\sigma(G)\ge k$ with respect to $A$ if and only if there exists a
%   bipartite subgraph $G'$ of $G$ containing $A$ such that
%   \begin{itemize}
%   \item For all $a\in A,$ the valence of $a$ in $G'$ is $k+1$, and
%   \item $G'$ has surplus $\sigma(G')=k$ with respect to $A$.
%   \end{itemize}
% \end{prop}
% \begin{proof}
%   One direction is immediate --- if such a subgraph exists, it is
%   immediately that $\sigma(G)\ge k.$ For the other direction, let $G'$
%   be a minimal subgraph of $G$ which contains $A$ and has surplus
%   $k$. Then by Lemma 1.3.7 of \cite{LovaszPlummer}, for all $a\in A,$
%   the valence of $a$ in $G'$ is $k+1.$
% \end{proof}

\begin{question}\label{question:SurplusNonzero}
  Proposition \ref{prop:SurplusUpperBound} implies that if
  $\sigma(G_{\mathcal{T}})<3$, then $d_{\mathcal{T}}=0.$ Does the
  converse hold? That is, if $\sigma(G_{\mathcal{T}})=3$, is
  $d_{\mathcal{T}}$ necessarily nonzero? Experimentally, this appears
  true. The special case where $\bigcap_{e\in E}e\ne\emptyset$ follows
  from \cite{CastravetTevelev2013}, see Section \ref{sec:Hypertrees}.
\end{question}

\begin{rem}
  One obtains a very simple upper bound for
  $P(G_{\mathcal{T}}-\{v_1,v_2,v_3\})$, and thus for
  $d_{\mathcal{T}}$, via the Bregman-Minc inequality, which in this
  case says
  \begin{align}\label{eq:BregmanMinc}
  P(G_{\mathcal{T}}-\{v_1,v_2,v_3\})\le\prod_{e\in
    E}(\abs{e\setminus\{v_1,v_2,v_3\}}!)^{1/\abs{e\setminus\{v_1,v_2,v_3\}}}.
  \end{align}
  For example, in Example \ref{ex:matchings}, the bound
  \eqref{eq:BregmanMinc} is
  $(1!)^{1/1}\cdot(2!)^{1/2}\cdot(3!)^{1/3}\approx2.5698$. This is
  sharp in this case in the sense that $d_{\mathcal{T}}=2$, but
  \eqref{eq:BregmanMinc} appears to be very far from sharp in
  general. Note also that~\eqref{eq:BregmanMinc} and
  $\abs{e\setminus\{v_1,v_2,v_3\}}\le4$ imply the (very loose) uniform
  bound $d_{\mathcal{T}}\le 24^{(n-3)/4}=(2.21336\ldots)^{n-3}.$ In
  fact, the approach via intersection theory on $\Mbar_{0,n}$
  mentioned in Section \ref{sec:IntersectionTheory} can be used to
  slightly strengthen this uniform bound to
  $d_{\mathcal{T}}\le2^{n-5}$ for $n\ge5$.
\end{rem}

\subsection*{Acknowledgements} I am grateful to Rohini Ramadas, Melody
Chan, and David Speyer for helpful discussions, to Christoph Goldner
for explaining the context behind \cite{Goldner2021}, and to Ana-Maria
Castravet for pointing out the connection to
\cite{CastravetTevelev2013}. This project was supported by NSF Grant
DMS-1645877, and by a Zelevinsky Postdoctoral Fellowship at
Northeastern University.

\bibliographystyle{alpha} \bibliography{../../research.bib}

\end{document}